\documentclass{amsart}

\usepackage[utf8]{inputenc}
\usepackage[english]{babel}
\usepackage{tikz}
\usepackage{amsmath,amsfonts,amssymb,amscd}
\usepackage{hyperref}
\usepackage{a4wide}

\hypersetup{
    unicode=true,          
}

\author{Philippe Biane and Matthieu Josuat-Vergès}
\title[Minimal factorizations of a cycle]{Minimal factorizations of a cycle: a multivariate generating function}
\address{Laboratoire d'Informatique Gaspard Monge, Université Paris-Est Marne-la-Vallée, CNRS}
\thanks{This work is supported by ANR CARMA (ANR-12-BS01-0017).}

\newcommand{\nsqsubset}{ \sqsubset\hspace{-1.8mm}/\hspace{1mm} }

\DeclareMathOperator{\wt}{wt}
\newtheorem{theo}{Theorem}
\newtheorem{coro}[theo]{Corollary}
\newtheorem{lemm}[theo]{Lemma}
\newtheorem{prop}[theo]{Proposition}

\newtheorem{defi}[theo]{Definition}

\begin{document}

\tikzset{every picture/.style={scale=0.4}}
 
\begin{abstract}
It is known that the number of minimal factorizations of the long cycle in the symmetric group into a 
product of $k$ cycles of given lengths has a very simple formula: it is $n^{k-1}$ where $n$ is the rank 
of the underlying symmetric group and $k$ is the number of factors. In particular, this is $n^{n-2}$ 
for transposition factorizations. The goal of this work is to prove a multivariate generalization of 
this result.
As a byproduct, we get a multivariate analog of Postnikov's hook length formula for trees,
and a refined enumeration of final chains of noncrossing partitions.
\end{abstract}
 
\maketitle

\section{Introduction}

Let $c$ be the long cycle $(1, 2, 3, \dots, n)$ in the symmetric group $\mathfrak{S}_n$. 
It is elementary to see that at least $n-1$ factors are needed to write $c$ as a product of transpositions,
such as $c=(1,2)(2,3)\dots (n-1,n)$.
So, a factorization 
\[
  c = t_1 \dots t_{n-1}
\]
where each $t_i$ is a transposition is called {\it minimal}.
The number of minimal factorizations of the cycle $c$ is $n^{n-2}$, as was first shown by Dénes~\cite{denes}.

One can interpret this result as the counting of the number of maximal chains in the lattice of noncrossing partitions  of $[1,n]$, whose definition is recalled in the text below.
The interval partitions, consisting of partitions of $[1,n]$ whose parts are intervals, form a sublattice of the noncrossing partitions, isomorphic to the Boolean lattice of subsets of $[1,n-1]$, whose number of maximal chains is easily seen to be $(n-1)!$. One of the results of this paper is a formula interpolating between these two, namely we give a generating function for maximal chains $\pi_0<\pi_1<\dots<\pi_{n-1}$ in the noncrossing partition lattice of the form
$$\sum_{\pi_0,\pi_1,\dots\pi_{n-1}} \wt(\pi_0,\pi_1,\dots,\pi_{n-1})=\prod_{i=1}^{n-2}(iX_i+n-i)$$
where the weight $\wt$ is a monomial in the $X_i$ and is equal to 1 exactly when $\pi_0,\pi_1,\ldots\pi_{n-1}$ is a maximal chain of  interval partitions.

Let $\underline a=(a_1,\dots,a_r)$ where $a_i\geq 2$. A factorization $c=z_1 \dots z_r$ 
where $z_i$ is a cycle of length $a_i$ is said to be {\it of type} $\underline a$.
This exists only if $\sum_{i=1}^r (a_i-1)\geq n-1$, and the factorization is called {\it minimal}
in case of equality. We only consider minimal factorizations here and assume $\sum_{i=1}^r (a_i-1)=n-1$
in the sequel.
The first author \cite{biane} showed that the number of minimal factorizations of type $\underline a$ is $n^{r-1}$,
in particular it only depends on $r$. This was independently obtained by Du and Liu \cite{duliu} 
(see also Irving~\cite{irving}, Springer~\cite{springer} and the far-reaching generalizations by 
Krattenthaler and Müller~\cite{krattenthaler}).

We denote $\mathcal{M}(\underline a)$ the set of minimal factorizations of $c$ of type $\underline a$.
Again one can interpret such factorizations as chains of a certain type in the lattice of noncrossing partitions and extend the definition of the
  weight $\wt$ to obtain the following  
\begin{theo}\label{mainth}
Let $b_i=\sum_{j=1}^i (a_j-1)$. We have:
\begin{equation} \label{maineq}
    \sum_{z_1 \dots z_r \in\mathcal{M}(\underline a) }  \wt(z_1\dots z_r) = \prod_{i=1}^{r-1} \big( b_iX_i + n-b_i  \big).
\end{equation}
\end{theo}

In the context of finite Coxeter groups, Dénes' result is a particular case of Deligne's 
formula~\cite{deligne} that gives the number of reflection factorizations of a Coxeter element. 
Deligne's formula has also been interpreted as the number of maximal chains in the noncrossing 
partition lattice~\cite{chapoton}. A one parameter refinement of this enumeration has been obtained 
by the second author in  \cite{josuat}, and it is what naturally leads to the definition 
of the weight used to get the multivariate versions.

In the particular case of transposition factorizations ($a_i=2$ for all $i$),
Theorem~\ref{mainth} is equivalent to a multivariate hook length formula for trees.
This will be presented in Section~\ref{hookf}.

Again in the particular case of transposition factorizations, our result is in fact equivalent to a multivari-
ate enumeration of Cayley trees of Kreweras and Moszkowski~\cite{kreweasmoszkowski}. Indeed, our weights can be 
translated in terms of decreasing edges of trees. More generally, it is possible to adapt the proof in \cite{kreweasmoszkowski} 
to the case of decreasing edges in cacti \cite{springer}, thus giving an alternative full proof of Theorem~\ref{mainth}. 
We give some details in Section~\ref{seccacti}.

\section{Definitions and preliminaries}
\subsection{Some classes of partitions}
\label{defs}

Let $T$ be a finite totally ordered set.

A {\it noncrossing partition} of $T$ is a set partition such that 
there is no $i<j<k<\ell\in T$ with $i$ and $k$ in one block, and $j$,$\ell$ in another one. The noncrossing partitions of $T$ form a sublattice of the lattice of set partitions, for the refinement order where $\pi\leq \pi'$ if each block of $\pi$ is a subset of some block of $\pi'$. We denote it by $NC_T$, by $\hat 0$ the partition with $|T|$ blocks which is the smallest element in $NC_T$ and by $\hat 1$ the partition with one block, which is the largest.

An {\it interval partition} of $T$ is a  set partition whose blocks are intervals i.e.  sets of the form
$$[a,b]=\{t\in T;a\leq t\leq b\}\ \text{for}\ a,b\in T.$$ 
The interval partitions form a sublattice $I_T$ of $NC_T$. Let $t_m$ be the maximal element of $T$, to any interval partition one can associate  the set of maximal elements of its blocks, this gives a subset of $T$ which contains $t_m$ and conversely, any such subset comes from a unique interval partition. Moreover the order on interval partitions corresponds to the reverse inclusion order on subset of $T$ containing $t_m$. Taking the intersection with $T\setminus\{t_m\}$ gives an isomorphism between $I_T$ and the Boolean lattice of subsets of $T\setminus \{t_m\}$. In particular, the maximal chains in $I_T$ are in bijection with permutations of 
$T\setminus \{t_m\}$.

The blocks of an interval partition are totally ordered by comparing their elements.
Given an interval partition $(I_j)_{j\in \mathcal J}$ with at least three blocks, with smallest block $I_0$ and largest block $I_t$,  the partition obtained by merging $I_0$ and $I_t$ will be called a {\it near interval partition}.
If $T=\{1,\dots,n\}$ the near interval partitions are the partitions which are not interval partitions, but can be rotated by $i\mapsto i+k \mod(n)$ for some $k$ to be transformed into an interval partition.

In the sequel we will consider these definitions when $T$ is the set $\{1,\dots,n\}$ or a subset with the induced order relation.
\subsection{Embedding noncrossing partitions into the symmetric group}
Let $\pi$ be  a noncrossing partition of $T$. There exists a unique   a permutation $\sigma_\pi$ of $T$ whose orbits are the parts of $\pi$ and, if $i_1<i_2<\dots<i_r$ form a block of $\pi$, then $\sigma_\pi(i_k)=i_{k+1\mod(r)}$.
This defines an embedding $\pi\mapsto \sigma_\pi$ of $NC_T$ into the  group $S_T$ of permutations of $T$.
The image of this embedding can be characterized geometrically. For each permutation $\sigma$ of $T$ let 
$l(\sigma)=|T|-c(\sigma)$ where $c(\sigma)$ is the number of orbits of $\sigma$ then $l=l(\sigma)$ is the smallest length of a factorization 
 $\sigma=t_1\dots t_l$ into a product of transpositions.  It follows that  $l$ is a length function i.e. $d(\sigma,\tau)=l(\sigma\tau^{-1})$ defines a distance $d$ on the group $S_T$, the distance  in the Cayley graph, with vertex set $S_T$, such that  $(\sigma,\sigma')$ is an edge if and only if  $\sigma^{-1}\sigma'$ is a transposition. 
Let $C$ be the long cycle of $S_T$ which maps each element of $T$ to its successor and the largest element to the smallest one, then a permutation $\sigma$ is of the form $\sigma_\pi$ if and only if it lies on a geodesic for $d$ between the identity permutation $id$ and $C$ that is, if
$l(\sigma)+l(C\sigma^{-1})=l(C)=|T|-1$. 
The order relation on $NC_T$ can also be characterized geometrically: one has $\pi\leq\pi'$ if and only if 
$l(\sigma_{\pi'})=l(\sigma_\pi)+l(\sigma_\pi^{-1}\sigma_{\pi'})$ that is, if $\pi$ lies on a geodesic between $id$ and $\pi'$.

The following lemmas follow from the above geometric characterization.
\begin{lemm} Let $C=azb$ be a factorization with $l(C)=l(a)+l(z)+l(b)$ and $z$ is a cycle on the elements
$i_1<i_2\dots<i_r$ then $z(i_k)=i_{k+1\mod(r)}$ for $k=1,\dots,r$.
\end{lemm}

\begin{proof}
One has  $C=z(z^{-1}az)b$ with $l(C)=l(z)+l((z^{-1}az))+l(b)$ therefore $z$ is on a geodesic from $id $ to $C$, so it is of the form $\sigma_\pi$.
\end{proof}

\begin{lemm} Let $C=yz$ be a minimal factorization with a cycle $z$, then $y=\sigma_{\pi}$ where $\pi$ is an interval or near interval partition.
\end{lemm}
\begin{proof}
Follows easily from the previous lemma.
\end{proof}

More generally one has:
\begin{lemm}\label{near} Let $\pi$ be a noncrossing partition and $\sigma_\pi=\sigma_{\pi'}z$ be a minimal factorization with a cycle $z$ on $k$ elements,
then $\pi'$ is obtained from $\pi$ by splitting a block of $\pi$ into an interval or  near interval partition with $k$ blocks.
\end{lemm}

It is immediate to check, using the above properties that, if $C=t_1\dots t_{n-1}$ is a minimal factorization into transpositions, then one has
$\sigma_{\pi_0}=id,\sigma_{\pi_1}=t_1, \sigma_{\pi_2}=t_1t_2,\dots,\sigma_{\pi_{n-1}}=t_1\dots t_{n-1}=C$ where
$(\pi_0,\dots,\pi_{n-1})$ is a maximal chain in $NC_T$. 

\subsection{\texorpdfstring{Some classes of chains in $NC_T$}{Some classes of chains in NC T}}

\begin{defi}
We denote by $\mathcal{N}(\underline a)$ the set of $(r+1)$-tuple of noncrossing partitions $(\pi_0,\dots,\pi_r)$ such that:
\begin{itemize}
 \item $\pi_0 = \hat 0$ and $\pi_r=\hat 1$,
 \item $\pi_{i-1}$ is obtained from $\pi_i$ by splitting a block $B$ of $\pi$ into $a_i$ blocks $B_1,\dots,B_{a_i}$,
       which form either an interval partition or a near interval partition of $B$.
\end{itemize}
\end{defi}

\begin{prop} \label{propbijmn}
The map $(\pi_0,\dots,\pi_r)\mapsto (\sigma_{\pi_0},\sigma_{\pi_0}^{-1}\sigma_{\pi_1},\dots,\sigma_{\pi_{r-1}}^{-1}\sigma_{\pi_r})$ is a bijection from $\mathcal{N}(\underline a)$ to 
$\mathcal{M}(\underline a)$. 
\end{prop}

\begin{proof}
This follows from Lemma \ref{near}.
\end{proof}

\begin{defi} For a sequence $(\pi_0,\dots,\pi_r)\in \mathcal{N}(\underline a)$
we write $\pi_{i-1}\sqsubset \pi_i$ in the case where the blocks $B_1,\dots,B_{a_i}$ form an interval
partition of $B$.
The weight of $\Pi = (\pi_0,\dots,\pi_r) \in \mathcal{N}(\underline a)$ is
\[
  \wt( \Pi ) = \prod_{\substack{ 1\leq i \leq r  \\ \pi_{i-1} \nsqsubset \pi_i }} X_i.
\]
\end{defi}

Using the bijection in Proposition~\ref{propbijmn}, this permits to define the weight function on
$\mathcal{M}(\underline a)$ that was used in Equation~\eqref{maineq}.

\section{Proof of Theorem~\ref{mainth}}

\begin{defi}
For an $r$-tuple $\underline a = (a_1,\dots,a_r)$ such that $a_i\geq2$ for all $i$ and
satisfying the minimality condition $\sum_{i=1}^r (a_i-1)=n-1$, we define 
\[
  P_{ \underline a } ( X_1, \dots , X_{r-1} ) = \sum_{ \Pi \in\mathcal{N}( \underline a) }  \wt( \Pi ).
\]
For such a $r$-tuple $\underline a = (a_1,\dots,a_r)$, we also define $\underline a' = (a_1,\dots, a_{r-2}, a_{r-1}+a_r-1) $.
\end{defi}

Note that $\underline a'$ also satisfies the minimality condition. Our aim is to prove:

\begin{prop}
\begin{equation} \label{mainrec}
   P_{ \underline a } ( X_1, \dots , X_{r-1} ) = P_{ \underline a'  } ( X_1, \dots , X_{r-2} ) \times \big( (n-a_r)X_{r-1} + a_r \big).
\end{equation}
\end{prop}

Indeed, the formula in Theorem~\ref{mainth} immediately follows by induction (it is clear that $P_{(n)} = 1$).
Note that we have $b_{r-1} = n-a_r$ from the definition of $b_i$
and the condition $\sum_{i=1}^r (a_i-1) = n-1$.

In order to prove the previous proposition, we define a map $\Psi : \mathcal{N}(\underline a) \to \mathcal{N}(\underline a') $ 
such that
\begin{equation} \label{relwt}
   \sum_{\Pi \in \Psi^{-1}(\Gamma)  } \wt( \Pi ) = \wt(\Gamma) \times \big( (n-a_r)X_{r-1} + a_r \big)
\end{equation}
for any $\Gamma\in \mathcal{N}(\underline a ')$, and summing over $\Gamma$ proves \eqref{mainrec}.

Let $\Pi = (\pi_0,\dots,\pi_r)\in\mathcal{N}(\underline a)$ and 
$\Pi' = (\pi_0,\dots,\pi_{r-2} , \pi_r) $. Note that $\Pi'$ might not be an element of $\mathcal{N}(\underline a')$.
In general, $\Psi(\Pi)$ will have the form $\sigma(\Pi') = (\sigma(\pi_0),\dots,\sigma(\pi_{r-2}) , \sigma(\pi_r)) $ 
for some $\sigma\in\mathfrak{S}_n$.
Moreover we require that the restriction of $\sigma$ to each block $B$ of $\pi_{r-2}$ is increasing.
Indeed, under these conditions we have:
\begin{lemm}
Let $1\leq i\leq r-2$, then there holds $\sigma(\pi_{i-1}) \sqsubset \sigma(\pi_i)$ if and only if $\pi_{i-1} \sqsubset \pi_i$.
\end{lemm}

\begin{proof}
 The noncrossing partition $\pi_{i-1}$ is obtained from $\pi_i$ by splitting a block $B$.
 Since $\pi_i$ is a refinement of $\pi_{r-2}$, we have $B\subset C$ for some $C\in\pi_{r-2}$,
 so the restriction of $\sigma$ on $B$ is increasing.
 Since the order is preserved, the condition of being an interval or a near interval partition is
 preserved too.
\end{proof}

If $1\leq i \leq r-3$, the previous lemma will ensure that $\wt(\Pi)$ and $\wt(\Psi(\Pi))$ both 
contain, or both don't contain, a factor $X_i$ .

We will also define a map $\beta:\mathcal{N}(\underline a) \to \{1,\dots,n\}$ and we prove that
$\Pi \mapsto (\Psi(\Pi),\beta(\Pi))$ is a bijection from 
$\mathcal{N}(\underline a)$ to $\mathcal{N}(\underline a') \times \{1,\dots,n\}$.
In the pictures, we represent $\beta(\Pi)=i$ by drawing a vertical bar drawn between two integers $i-1$ 
and $i$ (if $2\leq i \leq n$) or to the left of $1$ (if $i=1$). See Figure~\ref{bijcases} for examples.

To define $\Psi$ and $\beta$, we use some notations to build noncrossing partitions. Let $\pi$ and $\rho$ 
be noncrossing partitions, and $i,j>0$, then

\begin{itemize}
\item $\pi^{[i]} $ is $\pi$ where all labels are shifted up by $i$.

\item $\pi \oplus \rho = \pi \cup \rho^{[j]}$ if $\pi \in \mathcal{NC}_j $.

\item $(i,j) \curvearrowright \pi =  \{ \{1\dots,i\} \cup \{n-j+1,\dots,n\} \} \cup \pi^{[i]} $.

\item when there is no ambiguity, an integer $i$ denote the one block partition of size $i$.

\end{itemize}

For example, a near interval partition with four blocks can be written $(a,b)\curvearrowright( c \oplus d \oplus e )$
where $a,b,c,d,e>0$.

Let $B$ denote the block of $\pi_{r-1}$ that splits in $\pi_{r-2}$. Moreover,
$I$, $J$, $K$ denote arbitrary interval partition, and $a$, $b$, $c$, etc. are integers (or one-block partitions).
Then the definition is the following (see also Figure~\ref{bijcases}):
\begin{itemize}
 \item Case 1: $\pi_{r-1}$ and $\pi_{r-2}$ are both interval partitions.

       Then $\sigma$ is the identity permutation, and $\beta(\Pi) = \min B$.

 \item Case 2: $\pi_{r-2}$ is an interval partition but $\pi_{r-1}$ is not.
 
       Then $\sigma$ is the identity permutation, and $\beta(\Pi) = n-b+1$ where $b$
       is such that we can write  $\pi_{r-1} = (a,b) \curvearrowright I $.
 
 \item Case 3: $\pi_{r-1} = (a,b) \curvearrowright (I\oplus c \oplus J)$
       and $\pi_{r-2} = (a,b) \curvearrowright (I\oplus K \oplus J)$.
       
       Then $\sigma(\pi_{r-2}) = I \oplus (a+b) \oplus K \oplus J$.
       (Although we do not define $\sigma$ explicitly, there is only one canonical choice.)
       And $\beta(\Pi)=\min B$.

 \item Case 4: $\pi_{r-1} = (a,b) \curvearrowright I$, and $\pi_{r-2} = J \oplus (a',b') \curvearrowright I \oplus K$
       with $0<a'\leq a$ and $0<b'\leq b$.
       
       Then $\sigma(\pi_{r-2}) = J \oplus I \oplus (a'+b') \oplus K$,
       and the bar is placed between the $a'$ first dots and $b'$ last dots of the block of size $a'+b'$.
       
 \item Case 5 : $\pi_{r-1} = I\oplus a \oplus J$ and $\pi_{r-2} = I\oplus (b,c)\curvearrowright K \oplus J$.
 
       Then $\sigma(\pi_{r-2}) = (b,c)\curvearrowright ( I \oplus K \oplus J)$. If $I$ is nonempty, the bar
       is placed to the left of its last block. Otherwise, the bar is placed in leftmost position ($\beta(\Pi)=1$).
  
 \item Case 6: $\pi_{r-1} = (a,b) \curvearrowright I$, 
       and $\pi_{r-2} = (a',b') \curvearrowright (J\oplus I \oplus K)$
       with $0<a'\leq a$ and $0<b'\leq b$.
 
       Then $\sigma$ is the identity, and the bar is placed between $I$ and $K$.
  
 \item Case 7: $\pi_{r-1} = (a,b) \curvearrowright (I\oplus c \oplus J)$, 
       and $\pi_{r-2} = (a,b) \curvearrowright (I\oplus (d,e)\curvearrowright K \oplus J)$.
 
       Then $\sigma(\pi_{r-2}) = (a,b) \curvearrowright (I\oplus (d+e) \oplus J \oplus K) $,
       and the bar is placed between the $d$ first dots and $e$ last dots of the block of size $d+e$.

 \item Case 8: $\pi_{r-1} = (a,b) \curvearrowright I$, 
       and $\pi_{r-2} = (a',b') \curvearrowright (J\oplus (d,e)\curvearrowright I \oplus K)$
       with $0 < a' \leq a$ and $0 < b' \leq b$.

       Then $\sigma(\pi_{r-2}) = (a',b') \curvearrowright (J\oplus I \oplus (d+e) \oplus K) $,
       and the bar is placed between the $d$ first dots and $e$ last dots of the block of size $d+e$.

 \item Case 9 and 10:  This is when $\pi_{r-2}$ contains a block $C$ which is a union of three intervals (and no less).
       Let $a$, $b$, $c$, denote the length of these intervals. The other blocks of $\pi_{r-2}$ are arranged
       as a union of two interval partitions $I$ and $J$ (from left to right). 
       
       If $\pi_{r-1}$ is obtained by joining $C$ with the blocks in $I$ (Case 9), then
       $\sigma(\pi_{r-2}) = (a+b,c) \curvearrowright (I\oplus J) $,
       and the bar is placed to the right of the $a$th dot.
       
       If $\pi_{r-1}$ is obtained by joining $C$ with the blocks in $J$ (Case 10), then
       $\sigma(\pi_{r-2}) = (a,b+c) \curvearrowright (I\oplus J) $,
       and the bar is placed between to the left of the $c$th dot starting from the right. 
 \end{itemize}

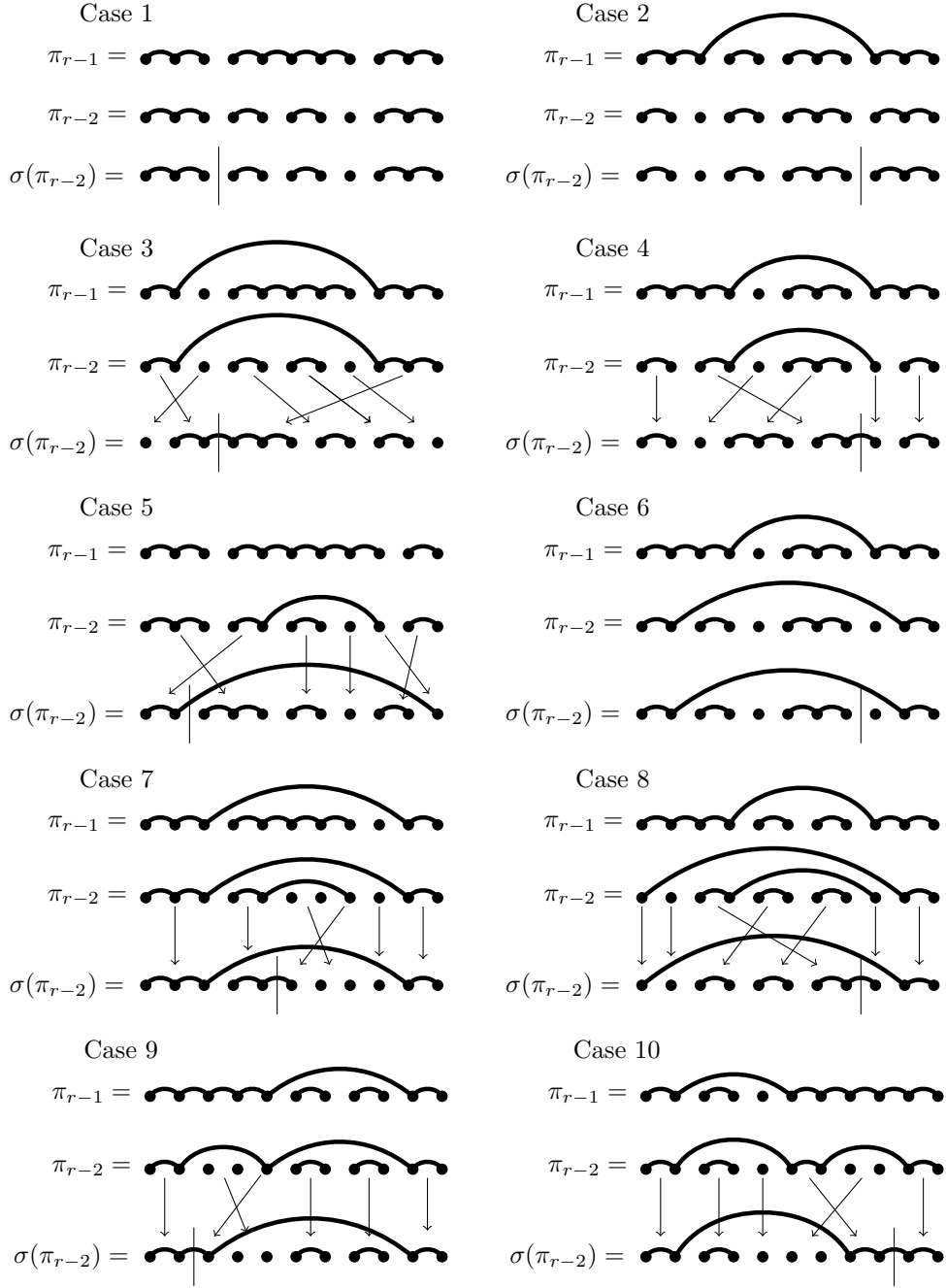
\begin{figure}[h!tp] \centering
 \begin{tikzpicture}
   \tikzstyle{edg} = [line width=0.6mm]
   \tikzstyle{ver} = [circle, draw, fill, inner sep=0.5mm]
   \node at (0,1.6) {Case 1};
   \node at (-1,0) {$\pi_{r-1}=$};
   \node[ver] at (1,0) {};
   \node[ver] at (2,0) {};
   \node[ver] at (3,0) {};
   \node[ver] at (4,0) {};
   \node[ver] at (5,0) {};
   \node[ver] at (6,0) {};
   \node[ver] at (7,0) {};
   \node[ver] at (8,0) {};
   \node[ver] at (9,0) {};
   \node[ver] at (10,0) {};
   \node[ver] at (11,0) {};
   \draw[edg] (1,0) to[bend left=60] (2,0);
   \draw[edg] (2,0) to[bend left=60] (3,0);
   \draw[edg] (9,0) to[bend left=60] (10,0);
   \draw[edg] (4,0) to[bend left=60] (5,0);
   \draw[edg] (5,0) to[bend left=60] (6,0);
   \draw[edg] (6,0) to[bend left=60] (7,0);   
   \draw[edg] (7,0) to[bend left=60] (8,0); 
   \draw[edg] (10,0) to[bend left=60] (11,0); 
  \begin{scope}[shift={(0,-2)}]
   \node at (-1,0) {$\pi_{r-2}=$};
   \node[ver] at (1,0) {};
   \node[ver] at (2,0) {};
   \node[ver] at (3,0) {};
   \node[ver] at (4,0) {};
   \node[ver] at (5,0) {};
   \node[ver] at (6,0) {};
   \node[ver] at (7,0) {};
   \node[ver] at (8,0) {};
   \node[ver] at (9,0) {};
   \node[ver] at (10,0) {};
   \node[ver] at (11,0) {};
   \draw[edg] (1,0) to[bend left=60] (2,0);
   \draw[edg] (2,0) to[bend left=60] (3,0);
   \draw[edg] (9,0) to[bend left=60] (10,0);
   \draw[edg] (4,0) to[bend left=60] (5,0);
   \draw[edg] (6,0) to[bend left=60] (7,0); 
   \draw[edg] (10,0) to[bend left=60] (11,0);   
  \end{scope}   
  \begin{scope}[shift={(0,-4)}]
   \node at (-1.7,0) {$\sigma(\pi_{r-2})=$};
   \node[ver] at (1,0) {};
   \node[ver] at (2,0) {};
   \node[ver] at (3,0) {};
   \node[ver] at (4,0) {};
   \node[ver] at (5,0) {};
   \node[ver] at (6,0) {};
   \node[ver] at (7,0) {};
   \node[ver] at (8,0) {};
   \node[ver] at (9,0) {};
   \node[ver] at (10,0) {};
   \node[ver] at (11,0) {};
   \draw[edg] (1,0) to[bend left=60] (2,0);
   \draw[edg] (2,0) to[bend left=60] (3,0);
   \draw[edg] (9,0) to[bend left=60] (10,0);
   \draw[edg] (4,0) to[bend left=60] (5,0);
   \draw[edg] (6,0) to[bend left=60] (7,0); 
   \draw[edg] (10,0) to[bend left=60] (11,0);   
   \draw      (3.5,1) to (3.5,-1);
   \end{scope}   
 \end{tikzpicture} 
 \hspace{0.5cm}
 \begin{tikzpicture}
   \tikzstyle{edg} = [line width=0.6mm]
   \tikzstyle{ver} = [circle, draw, fill, inner sep=0.5mm]
   \node at (0,1.6) {Case 2};
   \node at (-1,0) {$\pi_{r-1}=$};
   \node[ver] at (1,0) {};
   \node[ver] at (2,0) {};
   \node[ver] at (3,0) {};
   \node[ver] at (4,0) {};
   \node[ver] at (5,0) {};
   \node[ver] at (6,0) {};
   \node[ver] at (7,0) {};
   \node[ver] at (8,0) {};
   \node[ver] at (9,0) {};
   \node[ver] at (10,0) {};
   \node[ver] at (11,0) {};
   \draw[edg] (1,0) to[bend left=60] (2,0);
   \draw[edg] (2,0) to[bend left=60] (3,0);
   \draw[edg] (3,0) to[bend left=60] (9,0);
   \draw[edg] (9,0) to[bend left=60] (10,0);
   \draw[edg] (4,0) to[bend left=60] (5,0);
   \draw[edg] (6,0) to[bend left=60] (7,0);   
   \draw[edg] (7,0) to[bend left=60] (8,0); 
   \draw[edg] (10,0) to[bend left=60] (11,0); 
  \begin{scope}[shift={(0,-2)}]
   \node at (-1,0) {$\pi_{r-2}=$};
   \node[ver] at (1,0) {};
   \node[ver] at (2,0) {};
   \node[ver] at (3,0) {};
   \node[ver] at (4,0) {};
   \node[ver] at (5,0) {};
   \node[ver] at (6,0) {};
   \node[ver] at (7,0) {};
   \node[ver] at (8,0) {};
   \node[ver] at (9,0) {};
   \node[ver] at (10,0) {};
   \node[ver] at (11,0) {};
   \draw[edg] (1,0) to[bend left=60] (2,0);
   \draw[edg] (9,0) to[bend left=60] (10,0);
   \draw[edg] (4,0) to[bend left=60] (5,0);
   \draw[edg] (6,0) to[bend left=60] (7,0);   
   \draw[edg] (7,0) to[bend left=60] (8,0); 
   \draw[edg] (10,0) to[bend left=60] (11,0); 
  \end{scope}   
  \begin{scope}[shift={(0,-4)}]
   \node at (-1.7,0) {$\sigma(\pi_{r-2})=$};
   \node[ver] at (1,0) {};
   \node[ver] at (2,0) {};
   \node[ver] at (3,0) {};
   \node[ver] at (4,0) {};
   \node[ver] at (5,0) {};
   \node[ver] at (6,0) {};
   \node[ver] at (7,0) {};
   \node[ver] at (8,0) {};
   \node[ver] at (9,0) {};
   \node[ver] at (10,0) {};
   \node[ver] at (11,0) {};
   \draw[edg] (1,0) to[bend left=60] (2,0);
   \draw[edg] (9,0) to[bend left=60] (10,0);
   \draw[edg] (4,0) to[bend left=60] (5,0);
   \draw[edg] (6,0) to[bend left=60] (7,0);   
   \draw[edg] (7,0) to[bend left=60] (8,0); 
   \draw[edg] (10,0) to[bend left=60] (11,0); 
   \draw      (8.5,1) to (8.5,-1);
   \end{scope}   
 \end{tikzpicture} 
\vspace{0.2cm}
 
 \begin{tikzpicture}
   \tikzstyle{edg} = [line width=0.6mm]
   \tikzstyle{ver} = [circle, draw, fill, inner sep=0.5mm]
   \node at (0,1.6) {Case 3};
   \node at (-1,0) {$\pi_{r-1}=$};
   \node[ver] at (1,0) {};
   \node[ver] at (2,0) {};
   \node[ver] at (3,0) {};
   \node[ver] at (4,0) {};
   \node[ver] at (5,0) {};
   \node[ver] at (6,0) {};
   \node[ver] at (7,0) {};
   \node[ver] at (8,0) {};
   \node[ver] at (9,0) {};
   \node[ver] at (10,0) {};
   \node[ver] at (11,0) {};
   \draw[edg] (1,0) to[bend left=60] (2,0);
   \draw[edg] (2,0) to[bend left=60] (9,0);
   \draw[edg] (9,0) to[bend left=60] (10,0);
   \draw[edg] (4,0) to[bend left=60] (5,0);
   \draw[edg] (5,0) to[bend left=60] (6,0);   
   \draw[edg] (6,0) to[bend left=60] (7,0);   
   \draw[edg] (7,0) to[bend left=60] (8,0); 
   \draw[edg] (10,0) to[bend left=60] (11,0); 
  \begin{scope}[shift={(0,-2.5)}]
   \node at (-1,0) {$\pi_{r-2}=$};
   \node[ver] at (1,0) {};
   \node[ver] at (2,0) {};
   \node[ver] at (3,0) {};
   \node[ver] at (4,0) {};
   \node[ver] at (5,0) {};
   \node[ver] at (6,0) {};
   \node[ver] at (7,0) {};
   \node[ver] at (8,0) {};
   \node[ver] at (9,0) {};
   \node[ver] at (10,0) {};
   \node[ver] at (11,0) {};
   \draw[edg] (1,0) to[bend left=60] (2,0);
   \draw[edg] (2,0) to[bend left=60] (9,0);
   \draw[edg] (9,0) to[bend left=60] (10,0);
   \draw[edg] (4,0) to[bend left=60] (5,0);
   \draw[edg] (6,0) to[bend left=60] (7,0); 
   \draw[edg] (10,0) to[bend left=60] (11,0);   
  \end{scope}   
  \begin{scope}[shift={(0,-5.1)}]
   \draw[->] (1.5,2.3) to (2.5,0.7);
   \draw[->] (2.8,2.3) to (1.3,0.7);
   \draw[->] (4.7,2.3) to (6.6,0.7);
   \draw[->] (6.6,2.3) to (8.7,0.7);
   \draw[->] (6.6,2.3) to (8.7,0.7);
   \draw[->] (8.1,2.3) to (10.2,0.7);
   \draw[->] (9.8,2.3) to (5.8,0.7);
   \node at (-1.7,0) {$\sigma(\pi_{r-2})=$};
   \node[ver] at (1,0) {};
   \node[ver] at (2,0) {};
   \node[ver] at (3,0) {};
   \node[ver] at (4,0) {};
   \node[ver] at (5,0) {};
   \node[ver] at (6,0) {};
   \node[ver] at (7,0) {};
   \node[ver] at (8,0) {};
   \node[ver] at (9,0) {};
   \node[ver] at (10,0) {};
   \node[ver] at (11,0) {};
   \draw[edg] (2,0) to[bend left=60] (3,0);
   \draw[edg] (3,0) to[bend left=60] (4,0);
   \draw[edg] (4,0) to[bend left=60] (5,0);
   \draw[edg] (5,0) to[bend left=60] (6,0);
   \draw[edg] (7,0) to[bend left=60] (8,0);
   \draw[edg] (9,0) to[bend left=60] (10,0);   
   \draw      (3.5,1) to (3.5,-1);
   \end{scope}   
 \end{tikzpicture} 
  \hspace{0.5cm}
 \begin{tikzpicture}
   \tikzstyle{edg} = [line width=0.6mm]
   \tikzstyle{ver} = [circle, draw, fill, inner sep=0.5mm]
   \node at (0,1.6) {Case 4};
   \node at (-1,0) {$\pi_{r-1}=$};
   \node[ver] at (1,0) {};
   \node[ver] at (2,0) {};
   \node[ver] at (3,0) {};
   \node[ver] at (4,0) {};
   \node[ver] at (5,0) {};
   \node[ver] at (6,0) {};
   \node[ver] at (7,0) {};
   \node[ver] at (8,0) {};
   \node[ver] at (9,0) {};
   \node[ver] at (10,0) {};
   \node[ver] at (11,0) {};
   \draw[edg] (1,0) to[bend left=60] (2,0);
   \draw[edg] (2,0) to[bend left=60] (3,0);
   \draw[edg] (3,0) to[bend left=60] (4,0);
   \draw[edg] (4,0) to[bend left=60] (9,0);
   \draw[edg] (9,0) to[bend left=60] (10,0);
   \draw[edg] (10,0) to[bend left=60] (11,0); 
   \draw[edg] (6,0) to[bend left=60] (7,0);   
   \draw[edg] (7,0) to[bend left=60] (8,0); 
  \begin{scope}[shift={(0,-2.5)}]
   \node at (-1,0) {$\pi_{r-2}=$};
   \node[ver] at (1,0) {};
   \node[ver] at (2,0) {};
   \node[ver] at (3,0) {};
   \node[ver] at (4,0) {};
   \node[ver] at (5,0) {};
   \node[ver] at (6,0) {};
   \node[ver] at (7,0) {};
   \node[ver] at (8,0) {};
   \node[ver] at (9,0) {};
   \node[ver] at (10,0) {};
   \node[ver] at (11,0) {};
   \draw[edg] (1,0) to[bend left=60] (2,0);
   \draw[edg] (3,0) to[bend left=60] (4,0);
   \draw[edg] (4,0) to[bend left=60] (9,0);
   \draw[edg] (10,0) to[bend left=60] (11,0); 
   \draw[edg] (6,0) to[bend left=60] (7,0);   
   \draw[edg] (7,0) to[bend left=60] (8,0); 
  \end{scope}   
  \begin{scope}[shift={(0,-5.1)}]
   \draw[->] (1.5,2.3)  to (1.5,0.7);
   \draw[->] (3.6,2.3)  to (6.5,0.7);
   \draw[->] (4.8,2.3)  to (3.3,0.7);
   \draw[->] (6.8,2.3)  to (5.3,0.7);
   \draw[->] (10.5,2.3) to (10.5,0.7);
   \draw[->] (9,2.3)    to (9,0.7);
   \node at (-1.7,0) {$\sigma(\pi_{r-2})=$};
   \node[ver] at (1,0) {};
   \node[ver] at (2,0) {};
   \node[ver] at (3,0) {};
   \node[ver] at (4,0) {};
   \node[ver] at (5,0) {};
   \node[ver] at (6,0) {};
   \node[ver] at (7,0) {};
   \node[ver] at (8,0) {};
   \node[ver] at (9,0) {};
   \node[ver] at (10,0) {};
   \node[ver] at (11,0) {};
   \draw[edg] (1,0) to[bend left=60] (2,0);
   \draw[edg] (4,0) to[bend left=60] (5,0);
   \draw[edg] (5,0) to[bend left=60] (6,0);
   \draw[edg] (7,0) to[bend left=60] (8,0);
   \draw[edg] (8,0) to[bend left=60] (9,0);
   \draw[edg] (10,0) to[bend left=60] (11,0);
   \draw      (8.5,1) to (8.5,-1);
   \end{scope}   
 \end{tikzpicture} 
\vspace{0.2cm}
 
  \begin{tikzpicture}
   \tikzstyle{edg} = [line width=0.6mm]
   \tikzstyle{ver} = [circle, draw, fill, inner sep=0.5mm]
   \node at (0,1.6) {Case 5};
   \node at (-1,0) {$\pi_{r-1}=$};
   \node[ver] at (1,0) {};
   \node[ver] at (2,0) {};
   \node[ver] at (3,0) {};
   \node[ver] at (4,0) {};
   \node[ver] at (5,0) {};
   \node[ver] at (6,0) {};
   \node[ver] at (7,0) {};
   \node[ver] at (8,0) {};
   \node[ver] at (9,0) {};
   \node[ver] at (10,0) {};
   \node[ver] at (11,0) {};
   \draw[edg] (1,0) to[bend left=60] (2,0);
   \draw[edg] (2,0) to[bend left=60] (3,0);
   \draw[edg] (4,0) to[bend left=60] (5,0);
   \draw[edg] (5,0) to[bend left=60] (6,0);
   \draw[edg] (6,0) to[bend left=60] (7,0);   
   \draw[edg] (7,0) to[bend left=60] (8,0); 
   \draw[edg] (8,0) to[bend left=60] (9,0); 
   \draw[edg] (10,0) to[bend left=60] (11,0); 
  \begin{scope}[shift={(0,-2.5)}]
   \node at (-1,0) {$\pi_{r-2}=$};
   \node[ver] at (1,0) {};
   \node[ver] at (2,0) {};
   \node[ver] at (3,0) {};
   \node[ver] at (4,0) {};
   \node[ver] at (5,0) {};
   \node[ver] at (6,0) {};
   \node[ver] at (7,0) {};
   \node[ver] at (8,0) {};
   \node[ver] at (9,0) {};
   \node[ver] at (10,0) {};
   \node[ver] at (11,0) {};
   \draw[edg] (1,0) to[bend left=60] (2,0);
   \draw[edg] (2,0) to[bend left=60] (3,0);
   \draw[edg] (4,0) to[bend left=60] (5,0);
   \draw[edg] (5,0) to[bend left=60] (9,0); 
   \draw[edg] (6,0) to[bend left=60] (7,0); 
   \draw[edg] (10,0) to[bend left=60] (11,0);   
  \end{scope}   
  \begin{scope}[shift={(0,-5.5)}]
   \draw[->] (2.2,2.7)  to (3.7,0.7);
   \draw[->] (4.3,2.7)  to (1.8,0.7);
   \draw[->] (6.5,2.7)  to (6.5,0.7);
   \draw[->] (  8,2.7)  to (  8,0.7);
   \draw[->] (9.2,2.7)  to (10.7,0.7);
   \draw[->] (10.3,2.7) to (9.8,0.5);
   \node at (-1.7,0) {$\sigma(\pi_{r-2})=$};
   \node[ver] at (1,0) {};
   \node[ver] at (2,0) {};
   \node[ver] at (3,0) {};
   \node[ver] at (4,0) {};
   \node[ver] at (5,0) {};
   \node[ver] at (6,0) {};
   \node[ver] at (7,0) {};
   \node[ver] at (8,0) {};
   \node[ver] at (9,0) {};
   \node[ver] at (10,0) {};
   \node[ver] at (11,0) {};
   \draw[edg] (1,0) to[bend left=60] (2,0);
   \draw[edg] (2,0) to[bend left=40] (11,0);
   \draw[edg] (3,0) to[bend left=60] (4,0);
   \draw[edg] (4,0) to[bend left=60] (5,0);
   \draw[edg] (9,0) to[bend left=60] (10,0);
   \draw[edg] (6,0) to[bend left=60] (7,0); 
   \draw      (2.5,1) to (2.5,-1);
   \end{scope}   
 \end{tikzpicture}
\hspace{0.5cm}
 \begin{tikzpicture}
   \tikzstyle{edg} = [line width=0.6mm]
   \tikzstyle{ver} = [circle, draw, fill, inner sep=0.5mm]
   \node at (0,1.6) {Case 6};
   \node at (-1,0) {$\pi_{r-1}=$};
   \node[ver] at (1,0) {};
   \node[ver] at (2,0) {};
   \node[ver] at (3,0) {};
   \node[ver] at (4,0) {};
   \node[ver] at (5,0) {};
   \node[ver] at (6,0) {};
   \node[ver] at (7,0) {};
   \node[ver] at (8,0) {};
   \node[ver] at (9,0) {};
   \node[ver] at (10,0) {};
   \node[ver] at (11,0) {};
   \draw[edg] (1,0) to[bend left=60] (2,0);
   \draw[edg] (2,0) to[bend left=60] (3,0);
   \draw[edg] (3,0) to[bend left=60] (4,0);
   \draw[edg] (4,0) to[bend left=60] (9,0);
   \draw[edg] (9,0) to[bend left=60] (10,0);
   \draw[edg] (10,0) to[bend left=60] (11,0); 
   \draw[edg] (6,0) to[bend left=60] (7,0);   
   \draw[edg] (7,0) to[bend left=60] (8,0); 
  \begin{scope}[shift={(0,-2.5)}]
   \node at (-1,0) {$\pi_{r-2}=$};
   \node[ver] at (1,0) {};
   \node[ver] at (2,0) {};
   \node[ver] at (3,0) {};
   \node[ver] at (4,0) {};
   \node[ver] at (5,0) {};
   \node[ver] at (6,0) {};
   \node[ver] at (7,0) {};
   \node[ver] at (8,0) {};
   \node[ver] at (9,0) {};
   \node[ver] at (10,0) {};
   \node[ver] at (11,0) {};
   \draw[edg] (1,0) to[bend left=60] (2,0);
   \draw[edg] (3,0) to[bend left=60] (4,0);
   \draw[edg] (2,0) to[bend left=40] (10,0);
   \draw[edg] (10,0) to[bend left=60] (11,0); 
   \draw[edg] (6,0) to[bend left=60] (7,0);   
   \draw[edg] (7,0) to[bend left=60] (8,0); 
  \end{scope}   
  \begin{scope}[shift={(0,-5.5)}]
   \node at (-1.7,0) {$\sigma(\pi_{r-2})=$};
   \node[ver] at (1,0) {};
   \node[ver] at (2,0) {};
   \node[ver] at (3,0) {};
   \node[ver] at (4,0) {};
   \node[ver] at (5,0) {};
   \node[ver] at (6,0) {};
   \node[ver] at (7,0) {};
   \node[ver] at (8,0) {};
   \node[ver] at (9,0) {};
   \node[ver] at (10,0) {};
   \node[ver] at (11,0) {};
   \draw[edg] (1,0) to[bend left=60] (2,0);
   \draw[edg] (3,0) to[bend left=60] (4,0);
   \draw[edg] (2,0) to[bend left=40] (10,0);
   \draw[edg] (10,0) to[bend left=60] (11,0); 
   \draw[edg] (6,0) to[bend left=60] (7,0);   
   \draw[edg] (7,0) to[bend left=60] (8,0); 
   \draw      (8.5,1) to (8.5,-1);
   \end{scope}   
 \end{tikzpicture} 
\vspace{0.2cm}
 
 \begin{tikzpicture}
   \tikzstyle{edg} = [line width=0.6mm]
   \tikzstyle{ver} = [circle, draw, fill, inner sep=0.5mm]
   \node at (0,1.6) {Case 7};
   \node at (-1,0) {$\pi_{r-1}=$};
   \node[ver] at (1,0) {};
   \node[ver] at (2,0) {};
   \node[ver] at (3,0) {};
   \node[ver] at (4,0) {};
   \node[ver] at (5,0) {};
   \node[ver] at (6,0) {};
   \node[ver] at (7,0) {};
   \node[ver] at (8,0) {};
   \node[ver] at (9,0) {};
   \node[ver] at (10,0) {};
   \node[ver] at (11,0) {};
   \draw[edg] (1,0) to[bend left=60] (2,0);
   \draw[edg] (2,0) to[bend left=60] (3,0);
   \draw[edg] (3,0) to[bend left=40] (10,0);
   \draw[edg] (10,0) to[bend left=60] (11,0); 
   \draw[edg] (4,0) to[bend left=60] (5,0);   
   \draw[edg] (5,0) to[bend left=60] (6,0);   
   \draw[edg] (6,0) to[bend left=60] (7,0);   
   \draw[edg] (7,0) to[bend left=60] (8,0); 
  \begin{scope}[shift={(0,-2.5)}]
   \node at (-1,0) {$\pi_{r-2}=$};
   \node[ver] at (1,0) {};
   \node[ver] at (2,0) {};
   \node[ver] at (3,0) {};
   \node[ver] at (4,0) {};
   \node[ver] at (5,0) {};
   \node[ver] at (6,0) {};
   \node[ver] at (7,0) {};
   \node[ver] at (8,0) {};
   \node[ver] at (9,0) {};
   \node[ver] at (10,0) {};
   \node[ver] at (11,0) {};
   \draw[edg] (1,0) to[bend left=60] (2,0);
   \draw[edg] (2,0) to[bend left=60] (3,0);
   \draw[edg] (3,0) to[bend left=40] (10,0);
   \draw[edg] (10,0) to[bend left=60] (11,0);
   \draw[edg] (4,0) to[bend left=60] (5,0); 
   \draw[edg] (5,0) to[bend left=40] (8,0);   
  \end{scope}   
  \begin{scope}[shift={(0,-5.5)}]
   \node at (-1.7,0) {$\sigma(\pi_{r-2})=$};
   \draw[->] (2,2.7)  to (2,0.7);
   \draw[->] (4.5,2.7) to (4.5,1.2);
   \draw[->] (6.55,2.7)  to (7.3,0.7);
   \draw[->] (9,2.7)  to (9,1);
   \draw[->] (7.8,2.7)  to (6.3,0.7);
   \draw[->] (10.5,2.7)  to (10.5,0.9);
   \node[ver] at (1,0) {};
   \node[ver] at (2,0) {};
   \node[ver] at (3,0) {};
   \node[ver] at (4,0) {};
   \node[ver] at (5,0) {};
   \node[ver] at (6,0) {};
   \node[ver] at (7,0) {};
   \node[ver] at (8,0) {};
   \node[ver] at (9,0) {};
   \node[ver] at (10,0) {};
   \node[ver] at (11,0) {};
   \draw[edg] (1,0) to[bend left=60] (2,0);
   \draw[edg] (2,0) to[bend left=60] (3,0);
   \draw[edg] (3,0) to[bend left=40] (10,0);
   \draw[edg] (4,0) to[bend left=60] (5,0);
   \draw[edg] (5,0) to[bend left=60] (6,0);
   \draw[edg] (10,0) to[bend left=60] (11,0); 
   \draw      (5.5,1) to (5.5,-1);
   \end{scope}   
 \end{tikzpicture} 
\hspace{0.5cm}
 \begin{tikzpicture}
   \tikzstyle{edg} = [line width=0.6mm]
   \tikzstyle{ver} = [circle, draw, fill, inner sep=0.5mm]
   \node at (0,1.6) {Case 8};
   \node at (-1,0) {$\pi_{r-1}=$};
   \node[ver] at (1,0) {};
   \node[ver] at (2,0) {};
   \node[ver] at (3,0) {};
   \node[ver] at (4,0) {};
   \node[ver] at (5,0) {};
   \node[ver] at (6,0) {};
   \node[ver] at (7,0) {};
   \node[ver] at (8,0) {};
   \node[ver] at (9,0) {};
   \node[ver] at (10,0) {};
   \node[ver] at (11,0) {};
   \draw[edg] (1,0) to[bend left=60] (2,0);
   \draw[edg] (2,0) to[bend left=60] (3,0);
   \draw[edg] (3,0) to[bend left=60] (4,0);
   \draw[edg] (4,0) to[bend left=60] (9,0);
   \draw[edg] (5,0) to[bend left=60] (6,0);
   \draw[edg] (7,0) to[bend left=60] (8,0); 
   \draw[edg] (9,0) to[bend left=60] (10,0); 
   \draw[edg] (10,0) to[bend left=60] (11,0); 
  \begin{scope}[shift={(0,-2.5)}]
   \node at (-1,0) {$\pi_{r-2}=$};
   \node[ver] at (1,0) {};
   \node[ver] at (2,0) {};
   \node[ver] at (3,0) {};
   \node[ver] at (4,0) {};
   \node[ver] at (5,0) {};
   \node[ver] at (6,0) {};
   \node[ver] at (7,0) {};
   \node[ver] at (8,0) {};
   \node[ver] at (9,0) {};
   \node[ver] at (10,0) {};
   \node[ver] at (11,0) {};
   \draw[edg] (1,0) to[bend left=40] (10,0);
   \draw[edg] (3,0) to[bend left=60] (4,0);
   \draw[edg] (4,0) to[bend left=40] (9,0);
   \draw[edg] (5,0) to[bend left=60] (6,0); 
   \draw[edg] (7,0) to[bend left=60] (8,0); 
   \draw[edg] (10,0) to[bend left=60] (11,0); 
  \end{scope}   
  \begin{scope}[shift={(0,-5.5)}]
   \node at (-1.7,0) {$\sigma(\pi_{r-2})=$};
   \draw[->] (1,2.7) to (1,0.7);
   \draw[->] (2,2.7) to (2,1);
   \draw[->] (3.6,2.7) to (7,0.7);
   \draw[->] (5.3,2.7) to (3.8,0.7);
   \draw[->] (7.3,2.7) to (5.8,0.7);
   \draw[->] (9,2.7)   to ( 9,1);
   \draw[->] (10.5,2.7) to (10.5,0.7);
   \node[ver] at (1,0) {};
   \node[ver] at (2,0) {};
   \node[ver] at (3,0) {};
   \node[ver] at (4,0) {};
   \node[ver] at (5,0) {};
   \node[ver] at (6,0) {};
   \node[ver] at (7,0) {};
   \node[ver] at (8,0) {};
   \node[ver] at (9,0) {};
   \node[ver] at (10,0) {};
   \node[ver] at (11,0) {};
   \draw[edg] (1,0) to[bend left=40] (10,0);
   \draw[edg] (10,0) to[bend left=40] (11,0);
   \draw[edg] (3,0) to[bend left=60] (4,0);
   \draw[edg] (5,0) to[bend left=60] (6,0);
   \draw[edg] (7,0) to[bend left=60] (8,0);
   \draw[edg] (8,0) to[bend left=60] (9,0);
   \draw      (8.5,1) to (8.5,-1);
   \end{scope}   
 \end{tikzpicture} 
\vspace{0.2cm}
 
 \begin{tikzpicture}
   \tikzstyle{edg} = [line width=0.6mm]
   \tikzstyle{ver} = [circle, draw, fill, inner sep=0.5mm]
   \node at (0,1.6) {Case 9};
   \node at (-1,0) {$\pi_{r-1}=$};
   \node[ver] at (1,0) {};
   \node[ver] at (2,0) {};
   \node[ver] at (3,0) {};
   \node[ver] at (4,0) {};
   \node[ver] at (5,0) {};
   \node[ver] at (6,0) {};
   \node[ver] at (7,0) {};
   \node[ver] at (8,0) {};
   \node[ver] at (9,0) {};
   \node[ver] at (10,0) {};
   \node[ver] at (11,0) {};
   \draw[edg] (1,0) to[bend left=60] (2,0);
   \draw[edg] (2,0) to[bend left=60] (3,0);
   \draw[edg] (3,0) to[bend left=60] (4,0);
   \draw[edg] (4,0) to[bend left=60] (5,0);
   \draw[edg] (5,0) to[bend left=40] (10,0);
   \draw[edg] (10,0) to[bend left=60] (11,0); 
   \draw[edg] (6,0) to[bend left=60] (7,0);
   \draw[edg] (8,0) to[bend left=60] (9,0);
  \begin{scope}[shift={(0,-2.5)}]
   \node at (-1,0) {$\pi_{r-2}=$};
   \node[ver] at (1,0) {};
   \node[ver] at (2,0) {};
   \node[ver] at (3,0) {};
   \node[ver] at (4,0) {};
   \node[ver] at (5,0) {};
   \node[ver] at (6,0) {};
   \node[ver] at (7,0) {};
   \node[ver] at (8,0) {};
   \node[ver] at (9,0) {};
   \node[ver] at (10,0) {};
   \node[ver] at (11,0) {};
   \draw[edg] (1,0) to[bend left=60] (2,0);
   \draw[edg] (2,0) to[bend left=60] (5,0);
   \draw[edg] (5,0) to[bend left=40] (10,0);
   \draw[edg] (10,0) to[bend left=60] (11,0); 
   \draw[edg] (6,0) to[bend left=60] (7,0);   
   \draw[edg] (8,0) to[bend left=60] (9,0);
  \end{scope}   
  \begin{scope}[shift={(0,-5.5)}]
   \node at (-1.7,0) {$\sigma(\pi_{r-2})=$};
   \draw[->] (1.5,2.7)  to (1.5,0.7);
   \draw[->] (3.55,2.7)  to (4.3,0.85);
   \draw[->] (4.8,2.8)  to (3.2,0.7);
   \draw[->] (6.5,2.7)  to (6.5,0.7);
   \draw[->] (8.5,2.7)  to (8.5,0.7);
   \draw[->] (10.5,2.7)  to (10.5,0.9);
   \node[ver] at (1,0) {};
   \node[ver] at (2,0) {};
   \node[ver] at (3,0) {};
   \node[ver] at (4,0) {};
   \node[ver] at (5,0) {};
   \node[ver] at (6,0) {};
   \node[ver] at (7,0) {};
   \node[ver] at (8,0) {};
   \node[ver] at (9,0) {};
   \node[ver] at (10,0) {};
   \node[ver] at (11,0) {};
   \draw[edg] (1,0) to[bend left=60] (2,0);
   \draw[edg] (2,0) to[bend left=60] (3,0);
   \draw[edg] (3,0) to[bend left=40] (10,0);
   \draw[edg] (6,0) to[bend left=60] (7,0);
   \draw[edg] (8,0) to[bend left=60] (9,0);
   \draw[edg] (10,0) to[bend left=60] (11,0); 
   \draw      (2.5,1) to (2.5,-1);
   \end{scope}   
 \end{tikzpicture} 
\hspace{0.5cm}
 \begin{tikzpicture}
   \tikzstyle{edg} = [line width=0.6mm]
   \tikzstyle{ver} = [circle, draw, fill, inner sep=0.5mm]
   \node at (0,1.6) {Case 10};
   \node at (-1,0) {$\pi_{r-1}=$};
   \node[ver] at (1,0) {};
   \node[ver] at (2,0) {};
   \node[ver] at (3,0) {};
   \node[ver] at (4,0) {};
   \node[ver] at (5,0) {};
   \node[ver] at (6,0) {};
   \node[ver] at (7,0) {};
   \node[ver] at (8,0) {};
   \node[ver] at (9,0) {};
   \node[ver] at (10,0) {};
   \node[ver] at (11,0) {};
   \draw[edg] (1,0) to[bend left=60] (2,0);
   \draw[edg] (2,0) to[bend left=40] (6,0);
   \draw[edg] (3,0) to[bend left=60] (4,0);
   \draw[edg] (6,0) to[bend left=60] (7,0); 
   \draw[edg] (7,0) to[bend left=60] (8,0); 
   \draw[edg] (8,0) to[bend left=60] (9,0); 
   \draw[edg] (9,0) to[bend left=60] (10,0); 
   \draw[edg] (10,0) to[bend left=60] (11,0); 
  \begin{scope}[shift={(0,-2.5)}]
   \node at (-1,0) {$\pi_{r-2}=$};
   \node[ver] at (1,0) {};
   \node[ver] at (2,0) {};
   \node[ver] at (3,0) {};
   \node[ver] at (4,0) {};
   \node[ver] at (5,0) {};
   \node[ver] at (6,0) {};
   \node[ver] at (7,0) {};
   \node[ver] at (8,0) {};
   \node[ver] at (9,0) {};
   \node[ver] at (10,0) {};
   \node[ver] at (11,0) {};
   \draw[edg] (1,0) to[bend left=60] (2,0);
   \draw[edg] (2,0) to[bend left=60] (6,0);
   \draw[edg] (3,0) to[bend left=60] (4,0);
   \draw[edg] (6,0) to[bend left=60] (7,0);
   \draw[edg] (7,0) to[bend left=60] (10,0);
   \draw[edg] (10,0) to[bend left=60] (11,0);
  \end{scope}   
  \begin{scope}[shift={(0,-5.5)}]
   \node at (-1.7,0) {$\sigma(\pi_{r-2})=$};
   \draw[->] (1.5,2.7) to (1.5,0.7);
   \draw[->] (3.5,2.7) to (3.5,0.7);
   \draw[->] (5,2.7) to (5,0.7);
   \draw[->] (6.6,2.7) to (8.2,0.7);
   \draw[->] (8.4,2.7) to (6.7,0.7);
   \draw[->] (10.5,2.7) to (10.5,0.7);
   \node[ver] at (1,0) {};
   \node[ver] at (2,0) {};
   \node[ver] at (3,0) {};
   \node[ver] at (4,0) {};
   \node[ver] at (5,0) {};
   \node[ver] at (6,0) {};
   \node[ver] at (7,0) {};
   \node[ver] at (8,0) {};
   \node[ver] at (9,0) {};
   \node[ver] at (10,0) {};
   \node[ver] at (11,0) {};
   \draw[edg] (1,0) to[bend left=60] (2,0);
   \draw[edg] (2,0) to[bend left=60] (8,0);
   \draw[edg] (3,0) to[bend left=60] (4,0);
   \draw[edg] (8,0) to[bend left=60] (9,0);
   \draw[edg] (9,0) to[bend left=60] (10,0);
   \draw[edg] (10,0) to[bend left=60] (11,0);
   \draw      (9.5,1) to (9.5,-1);
   \end{scope}   
 \end{tikzpicture}  
 \caption{The map $\Psi$.\label{bijcases}}
\end{figure}

We can check that $\pi_{r-2} \sqsubset \pi_{r-1} $ if and only if $\sigma(\pi_{r-2})$ is an 
interval partition (both conditions are true in cases 1--4, but none is true in cases 5--10).
This means that $\wt(\Pi)$ and $\wt(\Psi(\Pi))$ both contain or don't contain a factor $X_{r-2}$.
So at this point, we have proved that $\wt(\Psi(\Pi)) = \wt(\Pi) $ or $\wt(\Psi(\Pi)) = \wt(\Pi) \times X_{r-1} $.

\begin{lemm}
 Suppose we know $\sigma(\pi_{r-2})$ and the location of the bar, then we can deduce which one of 
 the 10 cases was applied.
\end{lemm}

\begin{proof}
Suppose for example that $\sigma(\pi_{r-2})$ is a near interval partition, and the bar is between $i$ and $i+1$ where both 
of these integers are in the non-interval block of $\sigma(\pi_{r-2})$. 
By examining the various cases in Figure~\ref{bijcases}, we see that we are in case 9 or 10 depending on whether 
the bar is on the left or right part of this non-interval block. It remains only to distinguish cases 1--8.

\begin{itemize}
 \item We can separate cases 1--4 from 5--8 by seeing whether $\sigma(\pi_{r-2})$ is an interval partition, or not.
 \item We can separate cases 3,4,7,8 from cases 1,2,5,6 by seeing whether the bar is between two integers that
       are in the same block of $\sigma(\pi_{r-2})$, or not. 
 \item We can separate cases 1,3,5,7 from 
       cases 2,4,6,8 by seeing whether the number of blocks of $\sigma(\pi_{r-2})$ entirely to the right of the 
       bar is $\geq a_{r-1}$ or $ < a_{r-1}$.
\end{itemize}
So the 3 criterions permits to distinguish the $2^3=8$ remaining cases.
\end{proof}

The previous lemma implies that $(\Psi,\beta)$ is a bijection, because in each given case we can recover
$\pi_{r-2}$ and $\pi_{r-1}$ from $\sigma(\pi_{r-2})$ and the location of the bar.
So for a given $\Gamma\in\mathcal{N}(\underline a')$, the $n$ elements in $\Psi^{-1}(\Gamma)$
can be obtained from the $n$ possible locations of the bar.

In order to get the factor $\big( (n-a_r)X_{r-1} + a_r \big)$ in \eqref{relwt}, it remains to prove 
the following : 

\begin{lemm}
If $\sigma(\pi_{r-2})$ is given, among the $n$ possible locations 
of the bar, there are exactly $a_r$ that result in $\pi_{r-1}$ being an interval partition.
\end{lemm}

\begin{proof}
Suppose first that $\sigma(\pi_{r-2})$ is an interval partition. The only case where
$\pi_{r-1}$ is an interval partition is case 1. This means that the bar need to be located 
just to the left of some block, and need to have at least $a_{r-1}$ blocks to its right. 
Since the number of blocks in $\sigma(\pi_{r-2})$ is $a_{r-1}+a_{r}-1$, there are $a_r$ possible locations.

Now suppose that $\sigma(\pi_{r-2})$ is not an interval partition. The only case where
$\pi_{r-1}$ is an interval partition is case 5. Once again this means that the bar need to be located 
just to the left of some block, and need to have at least $a_{r-1}$ blocks to its right. 
In this case too we get $a_r$ possible locations.
\end{proof}

This completes the proof of Equation~\eqref{relwt}. As explained earlier, we deduce Theorem~\ref{mainth}.

\section{A hook formula for labelled trees}

\label{hookf}

In this section, we are in the particular case $\underline a=(2,\dots,2)$ where there are $n$ 2's (this corresponds
to counting transposition factorizations in $\mathfrak{S}_{n+1}$).
We give a mutivariate version of Postnikov's hook length formula \cite[Corollary~17.3]{postnikov} as a consequence
of Theorem~\ref{mainth}. The interpretation of this hook length formula using noncrossing chains has 
been made by the second author in \cite[Section~5]{josuat}, and along the same lines it gives the 
multivariate version we present here. 

Note that another mutivariate hook length formula generalizing Postnikov's
was obtained by Féray and Goulden \cite{feraygoulden}. It seems unrelated to ours (in the sense that
one would easily implies the other), but it is likely that their methods can give another
proof of our Theorem~\ref{hookformula} below.

Let us define a polynomial
\[
  P_n (X_0,\dots,X_{n-1})  = \sum_{\Pi\in\mathcal{N}(\underline a)} \wt(\Pi).
\]
By Theorem~\ref{mainth}, it is equal to $\prod_{i=1}^{n-1} \big( iX_i + (n+1-i) \big)$. Note that $P_n$
is considered as an $n$-variable polynomial. Introducing the (seemingly useless) variable $X_0$ 
makes more convenient to write the next lemma.

\begin{lemm} \label{lemmrecpn}
Let $n\geq2$, then we have
\begin{equation} \label{eqrecpn}
   P_n (X_0,\dots,X_{n-1}) = \frac{(n-1)X_{n-1}+2}{2} \sum  P_{\# I} (X_{i_1},X_{i_2},\dots ) P_{\# J}(X_{j_1},X_{j_2},\dots)
\end{equation}
where the sum is over $I$, $J$ such that $\{0,\dots,n-2\} = I \uplus J$,
and $I=\{i_1,i_2,\dots\}$, $J=\{j_1,j_2,\dots\}$. 
This is a recursion whose initial case is $P_1(X_0)=1$.
\end{lemm}

\begin{proof}
Let $\Pi=(\pi_0,\dots,\pi_{n})\in\mathcal{M}(\underline a)$, and denote by $B$ and $C$ the two blocks of $\pi_{n-1}$.
Because of the symmetry, it is convenient to say that $B$ is a distinguished block of $\pi_{n-1}$, and
after computing the generating function of such objects we divide the result by 2.
Then we consider:
\begin{align}
  \Pi_1 &= (\pi_0|_B, \dots, \pi_{n-1}|_B) , \\
  \Pi_2 &= (\pi_0|_C, \dots, \pi_{n-1}|_C) ,
\end{align}
where $\pi|_B = \{ X\in\pi \,: \, X\subset B \}$. 
Let $\Pi_1'$ (respectively, $\Pi_2'$) be what we obtain after removing the repeated entries in $\Pi_1$
(respectively, $\Pi_2$).
To encode the location of repeated entries we define:
\begin{align*}
   I &= \{ i \, : \,  \pi_{i}|_B =  \pi_{i+1}|_B  \}, \\
   J &= \{ i \, : \,  \pi_{i}|_C =  \pi_{i+1}|_C  \}
\end{align*}
We have $I\uplus J = \{0,\dots,n-2\} $, moreover $\Pi_1'$ (respectively $\Pi_2'$) is a maximal chain of $NC_B$
(respectively $NC_C$). These properties follows the fact that the interval $[\hat 0, \pi_{n-1} ]$ is
isomorphic to $NC_B\times NC_C$.

The map $\Pi \mapsto (B,C,I,J,\Pi_1,\Pi_2) $ is bijective and proves combinatorially Equation~\eqref{eqrecpn}.
Indeed, for fixed $B$ and $C$ we get the sum over $I$ and $J$.
To get the factor $(n-1)X_{n-1}+2$, observe that when $\#I$ and $\#J$ are fixed, all the possible $B$ and $C$
are obtained from each other by the cyclic rotation through $\{1,\dots,n+1\}$, and there are two interval
partitions in this orbit.
It remains only to divide by 2 for symmetry reasons as mentionned above.
\end{proof}

The recursion in the previous lemma is conveniently interpreted in terms of trees. See the discussion 
at the end of this section for more information about this particular kind of trees.

\begin{defi}
An \emph{André tree} on $n$ vertices is a labelled tree such that
\begin{itemize}
 \item each internal vertex has either one or two unordered descendants,
 \item vertices are labelled with integers from $1$ to $n$, decreasingly from the root to the leaves.
\end{itemize}
We denote $\mathcal{A}_n$ the set of André trees with $n$ vertices.
If $1\leq i \leq n$, we denote by $h_i(T)$ the {\it hook length} of the vertex with label $i$ in $T$,
i.e. the number of vertices below the vertex with label $i$ (including itself).
The \emph{weight} of an André tree $T\in\mathcal{T}_n$ is defined as
\[
    \wt(T) =  \prod_{\substack{ 1 \leq i \leq n  \\   h_i(T) >1  }}  \bigg( X_{i-1} \big(h_i(T) -1\big) +2 \bigg).
\]
\end{defi}

For example, the André trees on $4$ vertices are in Figure~\ref{tree5}.

\tikzstyle{every node}=[circle, draw, inner sep = 0.5mm, font=\footnotesize]
\tikzstyle{level}=[sibling distance = 7mm]

\tikzset{every picture/.style={scale=1}}
\begin{figure}[h!tp]  \centering
\begin{tikzpicture}[level distance = 6mm]
\node {4}
   child {node {3}
     child {node {2}
       child {node {1}
       }
     }
   };
\end{tikzpicture}
\hspace{1cm}
\begin{tikzpicture}[level distance = 6mm]
\node {4}
   child {node {3}
     child {node {2}
     }
     child {node {1}
     }
   } ;
\end{tikzpicture}
\hspace{1cm}
\begin{tikzpicture}[level distance = 6mm]
\node {4}
   child {node {3}
   }
   child {node {2}
     child {node {1}
     }
   };
\end{tikzpicture}
\hspace{1cm}
\begin{tikzpicture}[level distance = 6mm]
\node {4}
   child {node {2}
   }
   child {node {3}
     child {node {1}
     }
   };
\end{tikzpicture}
\hspace{1cm}
\begin{tikzpicture}[level distance = 6mm]
\node {4}
   child {node {1}
   }
   child {node {3}
     child {node {2}
     }
   };
\end{tikzpicture}
\caption{The André trees with 4 vertices. \label{tree5}}
\end{figure}
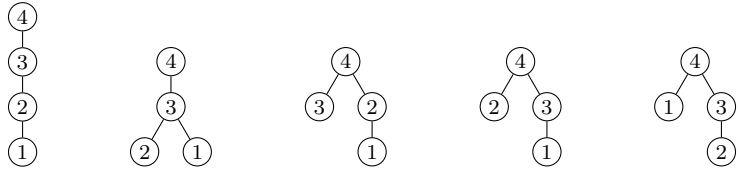

\begin{theo} \label{hookformula}
We have:
\[
  \sum_{ T \in \mathcal{T}_{n} }  \wt(T)  = \prod_{i=1}^{n-1} \big(  iX_i +   (n+1-i)  \big).
\] 
\end{theo}

\begin{proof}
We show that the left hand side satisfies the same recursion as $P_n$, as given in Lemma~\ref{lemmrecpn}.
The result is clear for $n=1$, so let $n\geq 2$.

For each $T \in \mathcal{T}_n$, the contribution of the root to $\wt(T)$ is a factor $(n-1)X_{n-1}+2$,
since its hook length is $n$. The rest of the tree is an unordered pair $\{U,V\}$ of trees,
one of them being possibly empty (in the case where the root has only one descendant). 
We can rather consider ordered pairs $(U,V)$ and divide the result by 2 at the end, since $U$ and $V$ 
can always be distinguished by the labels they contain. This gives the factor $\frac{(n-2)X_{n-2}+2}{2}$.

The labels of $U$ and $V$ form a partition $\{1,\dots,n-2 \} = I \uplus J$, and
the tree $U$ (respectively, $V$) is an element of $\mathcal{A}_{\#I}$ 
(respectively, $\mathcal{A}_{\#J}$) with an appropriate relabelling.
This permits to write the recursion. We omit details.
\end{proof}

For example, the hook formula with $n=4$ is as follows 
(where the five terms are in the order coming from the trees in Figure~\ref{tree5}):
\begin{align*}
  & (2+3X_3)(2+2X_2)(2+X_1) + (2+3X_3)(2+2X_2) \\
  & + (2+3X_3)(2+X_1) + (2+3X_3)(2+X_2) + (2+3X_3)(2+X_2)  \\
  = & \; (X_1+4)(2X_2+3)(3X_3+2).
\end{align*}

André trees were first introduced by Foata and Schützenberger \cite{foata}, who showed that 
$\# \mathcal{A}_n$ is the $n$th Euler number $E_n$ (which can be defined as the number 
of alternating permutations in $\mathfrak{S}_n$).
They were used by Stanley \cite{stanley} to show that $E_n$ is the number 
of orbits for the action of $\mathfrak{S}_{n+1}$ on maximal chains of set partitions
on $\{1,\dots,n+1\}$, so they were not unexpected in the present context.
Stanley's bijection explains why proving the above recursion is essentially
the same on chains of partitions or on André trees. But it also gives a fully 
combinatorial interpretation of the hook formula as follows.
Although $\mathcal{N}(\underline a)$ is not stable under the action of $\sigma_{n+1}$,
we can consider the equivalence relation $\sim$ on 
$\mathcal{N}(\underline a)$ defined by $\Pi_1\sim \Pi_2$ if there is $\sigma\in\mathfrak{S}_n$
such that $\sigma(\Pi_1)=\Pi_2$. It is easy to see that each orbit of maximal chains of partitions
contains a chain of noncrossing partitions, so the equivalence classes are indexed by André trees.
And the generating function of the equivalence class of index $T$ is $\wt(T)$. So the
hook express the fact that equivalence classes form a partition of $\mathcal{N}(\underline a)$.

We end this section by an open question. It would be very interesting if the recursion in 
Lemma~\ref{lemmrecpn} could be solved in a direct way leading to
$P_n(X_0,\dots,X_{n-2}) = \prod_{i=1}^{n-2} \big(  iX_i +   (n-i)  \big)$.
It would give an alternative proof of our multivariate hook formula, 
or equivalently, of the transposition case of Theorem~\ref{mainth}
The methods of \cite{feraygoulden} are quite likely to apply for this kind of problem.

\section{Final chains of noncrossing partitions}
\label{secfinalchains}

In this section, we present another interesting consequence of Theorem~\ref{mainth}.
We are still in the transposition case ($r=n-1$ and $a_i=2$ for all $i$).

\begin{defi}
A {\it final chain} of length $k$ in $NC_n$ is a $k$-tuple
of elements $(\pi_{n-k} , \dots , \pi_{n-1})$ such that $\pi_{n-k} \lessdot \dots \lessdot \pi_{n-1} = \hat 1$.
The weight of such a chain $\Pi= (\pi_{n-k} , \dots , \pi_{n-1})$ is
\[
  \wt( \Pi ) = \prod_{\substack{ n-k \leq i \leq n-2  \\ \pi_{i} \nsqsubset \pi_i+1 }} X_i.
\]
\end{defi}

It follows from the results of Krattenthaler and Müller \cite{krattenthaler} that
the number of final chains  of length $k$ in $NC_n$ is $n^{k-2} \binom{n}{k}$.
A multivariate analog can be obtained from Theorem~\ref{mainth}.

\begin{coro}
We have
\[
    \sum \wt(\Pi) =  \frac 1n \binom{n}{k} \prod_{i=n-k}^{n-2} \big( iX_i + (n-i) \big)
\]
where the sum is over final chains of length $k$ in $NC_n$.
\end{coro}

\begin{proof}
Let us consider the set $S$ of maximal chains $\pi_0 \lessdot \dots \lessdot \pi_{n-1} $ with the property that 
$\pi_0 \sqsubset \dots \sqsubset \pi_{n-k} $. The weight generating function of $S$ is obtained via a specialization 
of Theorem~\ref{mainth}, more precisely we take $r=n-1$, $a_i=2$ for all $i$, then $X_i=0$ for $1\leq i \leq n-k-1$.
This gives $\prod_{i=n-k}^{n-2} \big( iX_i + (n-i) \big) \times \prod_{i=k+1}^{n-1}i $.

Given a final chain $\pi_{n-k} \lessdot \dots \lessdot \pi_{n-1} = \hat 1$ of length $k$,
there are $(n-k)!$ ways to complete it into a maximal chain in $S$. This is a consequence
of Lemma~\ref{lem:des} below.
So the generating function for final chains of length $k$ is $\frac{1}{(n-k)!}$ times that of $S$.
The result follows.
\end{proof}

\begin{lemm} \label{lem:des}
 Let $\pi\in NC_n $, and $r$ denote its rank. There are exactly $r$ elements $\rho$ satisfying 
 $\rho \lessdot \pi$ and $\rho\sqsubset \pi$. 
\end{lemm}

\section{\texorpdfstring{Decreasing edges in Cayley trees and cacti}{Decreasing edges in Cayley trees and cacti}}

\label{seccacti}

The material in this section comes from a reviewer's report. See acknowledgements for details.

It is well known that there are $n^{n-2}$ Cayley trees on $n$ vertices. Each Cayley tree is considered to be
rooted at the vertex with label $n$, and edges are oriented towards the root. An edge is {\it decreasing} if it
is oriented from $i$ to $j$ with $i>j$. The weight of a Cayley tree is a square free monomial in variables
$X_1 , X_2 , \dots$ such that there is a factor $X_{j-1}$ iff there is a decreasing edge starting from the vertex with
label j. The following result comes from \cite{kreweasmoszkowski}:

\begin{theo}
The weight generating function of Cayley trees is $\prod_{i=1}^{n-2} (iX_i + n-i)$.
\end{theo}

In the same vein, it is possible to give an interpretation of the right-hand side of \eqref{maineq} using cacti (see \cite{springer}). 
These trees and cacti are in bijection with chains of noncrossing partitions, and the factorized generating function can 
be proved using codes similar to Prüfer codes.

\section*{Acknlowdgement}
We thank the anonymous reviewer who informed us that our work is related with decreasing edges in
Cayley trees and cacti, and in particular gave an alternative proof of our result with a different method. 
Only parts of this could be reproduced here, in Section~\ref{seccacti}.

\end{document}